\documentclass{amsart}
\usepackage{amsfonts, amssymb, wasysym}
\usepackage{multicol}
\usepackage{amsmath}
\usepackage{latexsym}
\usepackage{mathrsfs}
\usepackage{epsfig}
\usepackage{graphicx}

\newtheorem{theorem}{Theorem}
\newtheorem{lemma}{Lemma}
\newtheorem{corollary}{Corollary}
\usepackage{cyr}
\def\russianL{\text{\cyr L}}
\DeclareSymbolFont{mcyr}{OT1}{wncyr}{m}{n}
\DeclareMathSymbol\Lob{\mathop}{mcyr}{76}

\def\Rho{{\rm P}}

\begin{document}

\title{On volumes of hyperbolic Coxeter polytopes and quadratic forms}

\author{John G. Ratcliffe and Steven T. Tschantz}

\address{Department of Mathematics, Vanderbilt University, Nashville, TN 37240
\vspace{.1in}}

\email{j.g.ratcliffe@vanderbilt.edu}

\date{}

\begin{abstract}
In this paper, we compute the covolume of the group of units of the 
quadratic form $f_d^n(x) = x_1^2+x_2^2+\cdots +x_n^2-d\hspace{.01in}x_{n+1}^2$ with $d$ an odd, positive, square-free integer. 
Mcleod has determined the hyperbolic Coxeter fundamental domain of the reflection subgroup of the group 
of units of the quadratic form $f_3^n$.
We apply our covolume formula to compute the volumes of these hyperbolic Coxeter polytopes. 
\end{abstract}

\maketitle

\section{Introduction} 
In 1996, we computed the covolume of the group of units of the quadratic form 
$$f_d^n(x) = x_1^2+x_2^2+\cdots +x_n^2-d\hspace{.01in}x_{n+1}^2$$
with $d$ an odd, positive, square-free integer,  
in order to determine the volume of some hyperbolic Coxeter simplices, see \cite{RT97} and \S 5 of \cite{J-T}.  
We never got around to writing up our computation except for the case $d = 1$ in \cite{RT97}. 
Recently, several papers have been written concerning the group of units of the quadratic form $f_3^n$, 
and so we felt it was time to write up our more general computation. 
Our main result is Theorem 4 which gives an explicit formula for the covolume 
of the group of units of $f_d^n$. 

J. Mcleod \cite{M} has determined the reflection subgroup of the group of units of $f_3^n$ 
and has shown that it has finite index if and only if $n \leq 13$. 
As an application of our computation, we determine the volume of the hyperbolic Coxeter fundamental domain of the reflection subgroup of the group of units of the quadratic form $f_3^n$, for $n \leq 13$, described in \cite{M}. 
See Table 1 for the volumes of Mcleod's polytopes. 

M. Belolipetsky and V. Emery \cite{B-E} proved that for each odd dimension $n \geq 5$, there is a 
unique, orientable, noncompact, arithmetic, hyperbolic $n$-orbifold of smallest volume and determined its volume. 
V. Emery mentioned in \cite{E} that for $n \equiv 3\ \mathrm{mod} \ 4$, with $n \geq 7$,   
the corresponding arithmetic group is commensurable to the group of units of $f_3^n$.  
We determine the ratio between these two covolumes.

\section{Siegel's Covolume Formula for Unit Groups}  

Let $f$ be a quadratic form in $n+1$ real variables, with $n\geq 2$,  
that is equivalent over ${\mathbb R}$ to the {\it Lorentzian quadratic form}
$$x_1^2+\cdots +x_n^2-x_{n+1}^2.$$
Let $S$ be the matrix of the quadratic form $f$. 
We assume that all the entries of $S$ are integers and $d = |\mathrm{det}\, S|$ is an odd, square-free, positive integer. 

The {\it group of units} of the form $f$ is the group of all
$(n+1)\times (n+1)$ matrices $A$, with integral entries, such that
$A^tSA=S$. A unit of $f$ is said to be {\it positive} or {\it negative}
according as $A$ leaves invariant each of the two connected components of
$\{x\in {\mathbb R}^{n+1}: x^tSx<0\}$ or interchanges them.
The group of positive units of $f$ corresponds to a discrete group
$\Gamma$ of isometries of hyperbolic $n$-space
$$H^n\,=\,\{x\in {E}^{n+1}: x_1^2+\cdots +x_n^2-x_{n+1}^2=-1
                \ \ {\rm and}\ x_{n+1}>0\}$$
under the equivalence of $f$ with the Lorentzian quadratic form. 
For a discussion, see \S1 of \cite{J-T2}. 

Let $q$ be a positive integer and let ${\mathbb Z}_{q}={\mathbb Z}/q{\mathbb Z}$. 
Let $E_q(S)$ be the number of $(n+1)\times (n+1)$ matrices $A$ over ${\mathbb Z}_q$ 
such that $A^t S A = S$. 
By Formula 82 in Siegel's paper \cite{S36}, we have that
\begin{equation}  
{\rm vol}(H^n/\Gamma)\, = \, 4d^{\frac{n+2}{2}}
\prod^n_{k=1}\pi^{-\frac{k}{2}}\Gamma({\textstyle{\frac{k}{2}}})\,\cdot\,
\lim_{q\to\infty}{2^{\omega(q)}\frac{q^{\frac{n(n+1)}{2}}}{E_q(S)}}
\end{equation}
where $\Gamma({\textstyle{\frac{k}{2}}})$ is the gamma function evaluated at $k/2$, 
and $\omega(q)$ is the number of distinct prime divisors of $q$, 
and $q$ goes to infinity via the sequence $2!, 3!, 4! \ldots \,.$
In this paper a centered dot will delimit a product from the next factor. 

We now evaluate the limit in Siegel's volume formula. 
Let $q = p_1^{a_1}p_2^{a_2}\cdots p_r^{a_r}$ be the prime factorization of $q$. 
Now $E_q(S)$ is a multiplicative function of $q$ by Lemma 18 of Siegel's paper \cite{S35}, and so we have 
\begin{equation}
\frac{2^{\omega(q)} q^{\frac{n(n+1)}{2}}}{E_q(S)} = \mathop{\prod}_{i=1}^r\frac{2p_i^{\frac{a_in(n+1)}{2}}}{E_{p_i^{a_i}}(S)}.
\end{equation}
Henceforth $p$ is a prime number.  By Lemma 18 of \cite{S35}, we have 
\begin{equation}
\frac{2p^{\frac{an(n+1)}{2}}}{E_{p^a}(S)} = \left\{ \begin{array}{ll}
\frac{2p^{\frac{3n(n+1)}{2}}}{E_{p^3}(S)} & \mbox{if $p\, |\, 2d$ and $a\geq 3$}, \vspace{.05in}\\ 
\frac{2p^{\frac{n(n+1)}{2}}}{E_p(S)} & \mbox{if $p\nmid 2d$ and $a\geq 1$}.
\end{array}
\right. 
\end{equation}
Therefore, we have
\begin{equation}
\lim_{q\to\infty}2^{\omega(q)}\frac{q^{\frac{n(n+1)}{ 2}}}{E_q(S)} = 
\prod_{p\, |\, 2d}\frac{2p^{\frac{3n(n+1)}{2}}}{E_{p^3}(S)}
\prod_{p\, \nmid\, 2d} \frac{2p^{\frac{n(n+1)}{2}}}{E_p(S)}.
\end{equation}

First assume that $n$ is even and $p\nmid 2d$.  
Then by Lemma 18 of \cite{S35}, we have that 
\begin{equation}
\frac{2p^{\frac{n(n+1)}{2}}}{E_p(S)} = \prod_{k=1}^\frac{n}{2}(1-p^{-2k})^{-1}.
\end{equation}
Hence we have
\begin{equation}
\prod_{p\, \nmid\, 2d} \frac{2p^{\frac{n(n+1)}{2}}}{E_p(S)} = \prod_{k=1}^\frac{n}{2}\prod_{p\, \nmid\, 2d}(1-p^{-2k})^{-1}.
\end{equation}
Using the product formula for the Riemann zeta function $\zeta(2k)$, we have 
\begin{equation}
\prod_{p\, \nmid\, 2d} \frac{2p^{\frac{n(n+1)}{2}}}{E_p(S)} = 
\prod_{k=1}^\frac{n}{2}\left(\prod_{p\, |\, 2d}(1-p^{-2k})\right)\zeta(2k).
\end{equation}

Now assume that $n$ is odd and $p\nmid 2d$. 
Then by Lemma 18 of \cite{S35}, we have that 
\begin{equation}
\frac{2p^{\frac{n(n+1)}{2}}}{E_p(S)} =
 \left(1-\left( \frac{(-1)^{\frac{n-1}{2}}d}{p}\right) p^{-(\frac{n+1}{2})}\right)^{-1} \,
\prod_{k=1}^{\frac{n-1}{2}}(1-p^{-2k})^{-1}
\end{equation}
where $(\pm d/p)$ is a Legendre symbol. 
Define a fundamental discriminant $D$ by the formula
\begin{equation}
D = \left\{\begin{array}{ll}
(-1)^{\frac{n-1}{2}}d  & \mbox{ if $(-1)^{\frac{n-1}{2}}d \equiv 1$ mod 4}, \\
4(-1)^{\frac{n-1}{2}}d  & \mbox{ if $(-1)^{\frac{n-1}{2}}d \equiv 3$ mod 4}.
\end{array}\right. 
\end{equation}
For all odd primes $p$, we have 
\begin{equation}
\left( \frac{(-1)^{\frac{n-1}{2}}d}{p}\right) = \left( \frac{D}{p}\right).
\end{equation} 
Hence we have 
\begin{equation}
\prod_{p\, \nmid\, 2d} \frac{2p^{\frac{n(n+1)}{2}}}{E_p(S)} =
\prod_{p\, \nmid\, 2d}  \left(1-\left( \frac{D}{p}\right) p^{-(\frac{n+1}{2})}\right)^{-1}
\prod_{k=1}^{\frac{n-1}{2}}(1-p^{-2k})^{-1}.
\end{equation}

Consider the Dirichlet $L$-series 
\begin{equation}
L(s,D) = \sum_{k=1}^\infty \left(\frac{D}{k}\right)k^{-s}
\end{equation}
where $(D/k)$ is a Kronecker symbol.  
This series converges absolutely for $s > 1$. 
The Kronecker symbol $(D/k)$ is a completely multiplicative function of $k$ by Theorem 1.4.9 in Cohen \cite{Cohen}. 
By Theorem 11.7 in Apostol \cite{A}, this $L$-function has the product formula
\begin{equation}
L(s,D) = \prod_p\left(1-\left(\frac{D}{p}\right)p^{-s}\right)^{-1}.
\end{equation}
Using the product formulas for $\zeta(2k)$ and $L(\frac{n+1}{2},D)$, we obtain the formula
\begin{equation}
\prod_{p\, \nmid\, 2d} \frac{2p^{\frac{n(n+1)}{2}}}{E_p(S)} =
\left(1-\left(\frac{D}{2}\right)2^{-(\frac{n+1}{2})}\!\right) L({\textstyle\frac{n+1}{2}},D)
\!\prod_{k=1}^{\frac{n-1}{2}}\!\left(\prod_{p\, |\, 2d}(1-p^{-2k})\!\right)\!\zeta(2k).
\end{equation}
It remains only to compute $E_{p^3}(S)$ for each prime number $p$ dividing $2d$. 

\section{The computation of $E_8(S)$} 

From now on, we assume that $f = f_d^n$.  
Then $S$ is the $(n+1)\times (n+1)$ diagonal matrix $\mathrm{diag}(1,\ldots, 1,-d)$. 
Let $k$ be a positive integer, and let $\mathbb{Z}_k= \mathbb{Z}/k\mathbb{Z}$. 
Let $\mathrm{O}(n+1,\mathbb{Z}_k)$ be the group of $(n+1)\times(n+1)$ matrices $A$ 
over the ring $\mathbb{Z}_k$ such that $A^tA = I$. 
Let $J$ be the $(n+1)\times (n+1)$ diagonal matrix $\mathrm{diag}(1,\ldots, 1,-1)$. 
Let $\mathrm{O}(n,1;\mathbb{Z}_k)$ be the group of $(n+1)\times(n+1)$ matrices $A$ 
over the ring $\mathbb{Z}_k$ such that $A^tJA = J$. 

Let $\mathrm{O}(S, \mathbb{Z}_k)$ 
be the group of invertible $(n+1)\times(n+1)$ matrices $A$ over the ring $\mathbb{Z}_k$ 
such that $A^tSA = S$. 
Then $E_8(S)$ is the order of $\mathrm{O}(S, \mathbb{Z}_8)$, since $d$ is odd.   

Let $\mathrm{O}^\ast(n+1,\mathbb{Z}_2)$ be the image of $\mathrm{O}(n+1,\mathbb{Z}_4)$ in $\mathrm{O}(n+1,\mathbb{Z}_2)$
under the homomorphism induced by the ring homomorphism from $\mathbb{Z}_4$ to $\mathbb{Z}_2$, 
and let $\mathrm{O}^\ast(n,1;\mathbb{Z}_2)$ be the image of $\mathrm{O}(n,1;\mathbb{Z}_4)$ in 
$\mathrm{O}(n+1,\mathbb{Z}_2)$
under the homomorphism induced by the ring homomorphism from $\mathbb{Z}_4$ to $\mathbb{Z}_2$. 

We know that reduction modulo 2 maps $\mathrm{O}(S, \mathbb{Z}_8)$ into $\mathrm{O}^\ast(n,1;\mathbb{Z}_2)$ 
if $d \equiv  1 \ \mathrm{mod}\ 4$ and into $\mathrm{O}^\ast(n+1,\mathbb{Z}_2)$ if $d \equiv  -1\ \mathrm{mod}\ 4$, 
since the image of $\mathrm{O}(S, \mathbb{Z}_8)$ factors through $\mathrm{O}(S, \mathbb{Z}_4)$. 
We next show that $\mathrm{O}(S, \mathbb{Z}_8)$ maps onto the appropriate $\mathrm{O}^\ast$ subgroup of 
$\mathrm{O}(n+1,\mathbb{Z}_2)$ under reduction modulo 2. 

\begin{lemma} 
If $d \equiv  -1 \ \mathrm{mod}\ 4$, then $\mathrm{O}(S, \mathbb{Z}_8)$ maps onto $\mathrm{O}^\ast(n+1,\mathbb{Z}_2)$
under reduction modulo 2. 
\end{lemma}
\begin{proof}  
By the discussion on page 60 of \cite{RT97}, the group $\mathrm{O}^\ast(n+1,\mathbb{Z}_2)$ is generated 
by the permutation matrices and, if $n+1\geq 6$, a matrix $C$ which is the identity except for a lower corner 
$6 \times 6$ block with all $0$ diagonal and all $1$ off-diagonal entries. 

First assume that $d \equiv  -1 \ \mathrm{mod}\ 8$. 
Then $\mathrm{O}(S, \mathbb{Z}_8) = \mathrm{O}(n+1,\mathbb{Z}_8)$. 
The permutation matrices in $\mathrm{O}^\ast(n+1,\mathbb{Z}_2)$ obviously lift to $\mathrm{O}(n+1,\mathbb{Z}_8)$. 
If $n+1\geq 6$, the matrix $C$ lifts to $\mathrm{O}(n+1,\mathbb{Z}_8)$, since the identity part of $C$ lifts as is, 
and the $6\times 6$ block lifts to a block with all $2$ diagonal and all $1$ off-diagonal entries. 

Now assume that $d \equiv  3 \ \mathrm{mod}\ 8$. 
First we argue that all the permutation matrices lift to $\mathrm{O}(S, \mathbb{Z}_8)$. 
Clearly the permutations of the first $n$ coordinates lift. 
Thus it suffices to show that the transposition of the last two coordinates also lifts,  
and the lift is the $2\times 2$ block with diagonal 2, 6 and all 1 off-diagonal entries, that is, 
$$\left(\begin{array}{ll} 2 & 1 \\ 1 & 6\end{array}\right)\left(\begin{array}{rr} 1 & 0 \\ 0 & -3\end{array}\right)
\left(\begin{array}{ll} 2 & 1 \\ 1 & 6\end{array}\right) \equiv 
\left(\begin{array}{rr} 1 & 0 \\ 0 & -3\end{array}\right) \ \mathrm{mod} \ 8.$$
We also need our lower $6\times 6$ corner block to lift. 
The lift has diagonal entries $0,0,0,0,0,4$ and all 1 off-diagonal entries. 
\end{proof}

\begin{lemma} 
If $d \equiv  1 \ \mathrm{mod}\ 4$, then $\mathrm{O}(S, \mathbb{Z}_8)$ maps onto $\mathrm{O}^\ast(n,1;\mathbb{Z}_2)$
under reduction modulo 2. 
\end{lemma}
\begin{proof}
The proof for the case $d \equiv  1 \ \mathrm{mod}\ 8$ follows from Lemma 4 in \cite{RT97}.  
Now assume $d \equiv  -3 \ \mathrm{mod}\ 8$. 
Identify $\mathrm{O}^\ast(n,\mathbb{Z}_2)$ with the subgroup of $\mathrm{O}^\ast(n,1;\mathbb{Z}_2)$ fixing 
the last standard basis vector $e_{n+1}$.  
By the proof of Lemma 4 in \cite{RT97}, the group $\mathrm{O}^\ast(n,1;\mathbb{Z}_2)$ is generated 
by the subgroup $\mathrm{O}^\ast(n,\mathbb{Z}_2)$ and a matrix $C$ which is the identity except for a lower corner 
$4 \times 4$ block with all $0$ diagonal and all $1$ off-diagonal entries. 
The subgroup $\mathrm{O}^\ast(n,\mathbb{Z}_2)$ lifts to $\mathrm{O}(S, \mathbb{Z}_8)$ 
as in the case $d \equiv -1\ \mathrm{mod} \ 8$ in Lemma 1. 
The $4\times 4$ lower corner block of $C$ lifts to the $4 \times 4$ block with diagonal $2, 2, 2, 4$ and all 1 off-diagonal entries. 
\end{proof}

\begin{lemma} 
Let  $\eta: \mathrm{O}(S, \mathbb{Z}_8) \to \mathrm{O}(S, \mathbb{Z}_2)$ be the homomorphism induced by reduction modulo 2. 
The kernel of $\eta$ has order $2^{n^2+2n+2}$. 
\end{lemma}
\begin{proof}
Each matrix $M$ in $ \mathrm{O}(S, \mathbb{Z}_8)$ can be written uniquely in the form $N + 2U+4V$ 
where $N, U$, and $V$ are zero-one matrices and we are interested in counting the $M$ with $N = I$. 
Now observe that 
$$M^tSM = S + 2(U^tS + SU) + 4(U^tSU + V^tS + SV).$$
Using the fact that $4(-d) \equiv 4$ mod 8, the last equation simplifies to 
$$M^tSM = S + 2(U^tS + SU) + 4(U^tU + V^t + V).$$
First assume that $d \equiv 1$ mod 4.  
Then $2(-d) \equiv 2(-1)$ mod 8 and the last equation simplifies to 
$$M^tSM = S + 2(U^tJ + JU) + 4(U^tU + V^t + V).$$
The proof now proceeds as in the proof of Theorem 5 of \cite{RT97}. 

Now assume that $d \equiv -1$ mod 4. 
Then $2(-d) \equiv 2$ mod 8 and the next to last equation simplifies to 
$$M^tSM = S + 2(U^t + U) + 4(U^tU + V^t + V).$$
The proof now proceeds as in the proof of Theorem 5 of \cite{RT97} with the simplification that $W = 0$. 
\end{proof}

Let $\epsilon_2(k) = 1$ if $k$ is even and $0$ if $k$ is odd and define 
\begin{equation}
\alpha(n)\,=\, \prod^n_{k=1}\left(2^k-\epsilon_2(k)\right).
\end{equation}
Define
\begin{equation}
\beta(n)\,=\,2^{n-2}+2^{(n-2)/2}\cos(n\pi/4),
\end{equation}
and
\begin{equation}
\gamma(n)\,=\,2^n+2^{n/2}\cos(n\pi/4).
\end{equation}
Note that the function $f(n) = \cos(n\pi/4)$ is periodic with period 8.  The values of $f(n)$ for $n = 0,1, \ldots, 7$ 
are $1, 2^{-1/2}, 0, -2^{-1/2}, -1, -2^{-1/2}, 0, 2^{-1/2}$, respectively. 

\begin{theorem} 
The value of $E_8(S)$ is given by 
$$E_8(S) = \left\{\begin{array}{ll} 
2^{n^2 + 2n +2}\alpha(n)/\beta(n+1) & \mbox{if  $d \equiv -1\ \mathrm{mod\ 4}$,} \\
2^{n^2 + 2n +2}\alpha(n)/\gamma(n-1) & \mbox{if  $d \equiv 1\ \mathrm{mod\ 4}$.}
\end{array}\right.$$
\end{theorem}
\begin{proof}
By Lemmas 2 of  \cite{RT97}, the order of $\mathrm{O}^*(n+1,\mathbb{Z}_2)$ is $\alpha(n)/\beta(n+1)$, 
and by Lemma 3 of \cite{RT97}, the order $\mathrm{O}^*(n,1;\mathbb{Z}_2)$ is $\alpha(n)/\gamma(n-1)$. 
The Theorem now follows from Lemmas 1, 2, 3. 
\end{proof}

\begin{corollary} 
We have that
$$\frac{2\cdot 2^{\frac{3n(n+1)}{2}}}{E_8(S)} = 
\frac{2^{\frac{n-1}{2}}+\cos((n+(-1)^{\frac{d+1}{2}})\pi/4)}{2^{\frac{n+3}{2}}{\displaystyle \prod_{k=1}^{[\frac{n}{2}]}(1-2^{-2k})}}.$$
\end{corollary}
\begin{proof}
By Theorem 1, we have that 
$$\frac{2\cdot 2^{\frac{3n(n+1)}{2}}}{E_8(S)} = 
\frac{2\cdot 2^{\frac{3n(n+1)}{2}}(2^{n-1}+2^{\frac{n-1}{2}}\cos((n+(-1)^{\frac{d+1}{2}})\pi/4))}
{2^{n^2+2n+2}{\displaystyle\prod_{k=1}^n(2^k-\epsilon_2(k))}}.$$
If $n$ is odd, we have that 
\begin{eqnarray*}
{\displaystyle\prod_{k=1}^n(2^k-\epsilon_2(k))} 
& = & {\displaystyle 2\prod_{k=1}^{\frac{n-1}{2}}2^{2k+1}(2^{2k}-1)} \\
& = & {\displaystyle 2\prod_{k=1}^{\frac{n-1}{2}}2^{4k+1}(1-2^{-2k})} 
\ \ = \ \  2^{\frac{n(n+1)}{2}}{\displaystyle \prod_{k=1}^{\frac{n-1}{2}}(1-2^{-2k})}.
\end{eqnarray*}
If $n$ is even, we have that 
\begin{eqnarray*}
{\displaystyle\prod_{k=1}^n(2^k-\epsilon_2(k))} 
& = & {\displaystyle \prod_{k=1}^{\frac{n}{2}}2^{2k-1}(2^{2k}-1)} \\
& = & {\displaystyle \prod_{k=1}^{\frac{n}{2}}2^{4k-1}(1-2^{-2k})} 
\ \ = \ \  2^{\frac{n(n+1)}{2}}{\displaystyle \prod_{k=1}^{\frac{n}{2}}(1-2^{-2k})}.
\end{eqnarray*}
Therefore, we have that 
\begin{eqnarray*}
\frac{2\cdot 2^{\frac{3n(n+1)}{2}}}{E_8(S)} 
& = & \frac{ 2\cdot 2^{\frac{3n(n+1)}{2}}2^{\frac{n-1}{2}}(2^{\frac{n-1}{2}}+\cos((n+(-1)^{\frac{d+1}{2}})\pi/4))}{
2^{n^2+2n+2}\,2^{\frac{n(n+1)}{2}}{\displaystyle \prod_{k=1}^{[\frac{n}{2}]}(1-2^{-2k})}} \\
& = & \frac{2^{\frac{n-1}{2}}+\cos((n+(-1)^{\frac{d+1}{2}})\pi/4)}{2^{\frac{n+3}{2}}{\displaystyle \prod_{k=1}^{[\frac{n}{2}]}(1-2^{-2k})}}.
\end{eqnarray*}

\vspace{-.2in}
\end{proof}

\section{The computation of $E_{p^3}(S)$}  

Let $p$ be a prime number that divides the odd, positive, square-free integer $d$.  
In this section, we determine the value of $E_{p^3}(S)$. 

Let $a$ be a positive integer with $a \geq 2$, and let $M \in \mathrm{GL}(n+1,\mathbb{Z}_{p^a})$. Define
\begin{equation}
\alpha(M) = \mathrm{diag}(1,1,\ldots, 1,0) \cdot M,
\end{equation}
\begin{equation}
\beta(M) = \mathrm{diag}(0,0,\ldots, 0,1) \cdot M.
\end{equation}
Then $\alpha(M)$ is the matrix $M$ with its last row set to zero, and $\beta(M)$ is $M$ with its first $n$ rows set to zero. 
Note that $M = \alpha(M) + \beta(M)$.  

Define $\overline{\beta}(M) = \beta(M)$ reduced modulo $p^{a-1}$ with entries in the range $[0,p^{a-1})$. 
Define $\overline{M} = \alpha(M) + \overline{\beta}(M)$. 
Then $M-\overline{M} = p^{a-1} U$ for $U = \beta(U)$ with entries in the range $[0, p)$. 
Define 
\begin{equation}
\overline{\mathrm{O}}(S,\mathbb{Z}_{p^a}) = \{\overline{M} : M \in \mathrm{O}(S,\mathbb{Z}_{p^a})\}.
\end{equation}

\begin{lemma} 
Let $M \in \mathrm{GL}(n+1,\mathbb{Z}_{p^a})$ with $a \geq 2$.  Then $M \in \mathrm{O}(S,\mathbb{Z}_{p^a})$ 
if and only if $\overline{M} \in \mathrm{O}(S,\mathbb{Z}_{p^a})$.  
Moreover 
$$E_{p^a}(S) = |\mathrm{O}(S,\mathbb{Z}_{p^a})| = p^{n+1}|\overline{\mathrm{O}}(S,\mathbb{Z}_{p^a})|.$$
\end{lemma}
\begin{proof} 
Let $M$ be an $(n+1)\times (n+1)$ matrix over $\mathbb{Z}$ such that $M^tSM \equiv S$ mod $p^a$. 
Then $(\mathrm{det}(M))^2(-d) \equiv -d$ mod $p^a$.  Hence $(\mathrm{det}(M))^2 \equiv 1$ mod $p^{a-1}$,  
and so $\mathrm{det}(M)$ is invertible modulo $p^a$. Hence the reduction of $M$ modulo $p^a$ is invertible. 
Therefore $E_{p^a}(S) = |\mathrm{O}(S,\mathbb{Z}_{p^a})|$. 

Now suppose $M$ is an $(n+1)\times (n+1)$ matrix over $\mathbb{Z}_{p^a}$. 
Observe that 
\begin{eqnarray*}
M^t SM & = & (\overline{M} + p^{a-1}U)^t S(\overline{M} + p^{a-1}U) \\
& = & \overline{M}^t S\overline{M} + p^{a-1}U^t S\overline{M} + p^{a-1}\overline{M}^t SU + p^{2a-2}U^t SU \\ 
& = & \overline{M}^tS\overline{M} + p^{a-1}(-d)U^t\overline{M} + p^{a-1}(-d)\overline{M}^tU \\
& = & \overline{M}^tS\overline{M}.
\end{eqnarray*} 
Thus $M \in \mathrm{O}(S,\mathbb{Z}_{p^a})$ if and only if $\overline{M} \in \mathrm{O}(S,\mathbb{Z}_{p^a})$. 
In the equation $M-\overline{M} = p^{a-1} U$, there are $p^{n+1}$ choices for $U$.  
Hence for each $\overline{M}$ there are $p^{n+1}$ choices for $M$. 
Therefore $|\mathrm{O}(S,\mathbb{Z}_{p^a})| = p^{n+1}|\overline{\mathrm{O}}(S,\mathbb{Z}_{p^a})|$. 
\end{proof}

\begin{lemma} 
If $a \geq 2$, we have that $|\overline{\mathrm{O}}(S,\mathbb{Z}_{p^{a+1}})| 
= p^{\frac{n(n+1)}{2}}|\overline{\mathrm{O}}(S,\mathbb{Z}_{p^a})|$.

\end{lemma}
\begin{proof}
If $M \in \overline{\mathrm{O}}(S,\mathbb{Z}_{p^{a+1}})$, then 
$N = \overline{(M\ \mbox{mod}\ p^a)} \in \overline{\mathrm{O}}(S,\mathbb{Z}_{p^a})$.
We want to count 
$$M = N + p^a\alpha(U) + p^{a-1}\beta(U)$$
for $U$ mod $p$ and $N \in \overline{\mathrm{O}}(S,\mathbb{Z}_{p^a})$ such that 
$M \in \overline{\mathrm{O}}(S,\mathbb{Z}_{p^{a+1}})$. 
Observe that 
\begin{eqnarray*}
M^tSM-S & = & N^tSN - S \\
& + & p^a(\alpha(U)^tSN+ N^tS\alpha(U)) \\
& + & p^{a-1}(\beta(U)^tSN + N^tS\beta(U)) \\
& + & p^{2a-1}(\alpha(U)^tS\beta(U) + \beta(U)^tS\alpha(U)) \\
& + & p^{2a}(\alpha(U)^tS\alpha(U)) + p^{2a-2}(\beta(U)^tS\beta(U)) \\
& = & N^tSN - S \\
& + & p^a(\alpha(U)^tN+ N^t\alpha(U)) \\
& + & p^a(-d/p)(\beta(U)^tN + N^t\beta(U)). 
\end{eqnarray*}
Let $V = \alpha(U) + (-d/p)\beta(U)$.  
Then $\alpha(U) = \alpha(V)$ and $\beta(U) \equiv (-d/p)^{-1}\beta(V)\ \mathrm{mod}\ p$. 

Now $M^tSM = S$ if and only if 
$$V^tN + N^tV \equiv \frac{S-N^tSN}{p^a}\ \mathrm{mod}\ p$$
where we regard the entries of $N$ to be integers in the range $[0, p^a)$.  
Let $W = N^tV$ mod $p$.  Then $V = (N^t)^{-1}W$ mod $p$.  Note that $N^t$ is invertible mod $p$, 
since $N$ is invertible mod $p$ by Lemma 4. 
Now $M^tSM = S$ if and only if 
$$W^t + W \equiv \frac{S-N^tSN}{p^a}\ \mathrm{mod}\ p.$$
This congruence determines the diagonal of $W$ and below diagonal entries in terms 
of arbitrary above diagonal entries for $p^{n(n+1)/2}$ choices for $W$ with entries mod $p$. 
Thus we have $p^{n(n+1)/2}$ choices for $U$ and as many $M \in \overline{\mathrm{O}}(S,\mathbb{Z}_{p^{a+1}})$ 
for each $N\in \overline{\mathrm{O}}(S,\mathbb{Z}_{p^a})$. 
\end{proof}

\begin{lemma}  
We have that 
$|\overline{\mathrm{O}}(S,\mathbb{Z}_{p^2})| = 2p^{\frac{n(n+1)}{2}}|\mathrm{O}(n, \mathbb{Z}_p)|.$
\end{lemma}
\begin{proof}  
Let $M \in \overline{\mathrm{O}}(S,\mathbb{Z}_{p^2})$ and decompose $M$ into blocks
$$M = \left(\begin{array}{l|l}
M_0 &  v \\ \hline 
w^t & x
\end{array}\right)$$
with $v \in \mathbb{Z}_{p^2}^n$, $w \in \mathbb{Z}_p^n$ and $x \in \mathbb{Z}_p$. 

Observe that $M^tSM = S$ if and only if 
$$\left\{\begin{array}{llr} 
M_0^tM_0 - dww^t \hspace{-.05in}& \equiv & I \ \, \mathrm{mod}\ p^2,  \\
M_0^tv - dxw & \equiv & 0 \ \, \mathrm{mod}\ p^2, \\
v^tv - dx^2  & \equiv &\hspace{-.05in} -d \ \, \mathrm{mod}\ p^2.
\end{array}\right.$$
The above system of congruences implies that 
$$\left\{\begin{array}{llr} 
M_0^tM_0 \hspace{-.05in}& \equiv & I \ \, \mathrm{mod}\ p,  \\
v & \equiv & 0 \ \, \mathrm{mod}\ p, \\
x^2  & \equiv &1 \ \, \mathrm{mod}\ p.
\end{array}\right.$$
Let $N_0 = (M_0\ \mathrm{mod} p)$.  
Then $M_0 = N_0 + pU_0$ for $U_0$ mod $p$ and $N_0 \in \mathrm{O}(n, \mathbb{Z}_p)$. 
For each such $N_0$ choose $x = \pm 1 \ \mathrm{mod}\ p$ and $w \in \mathbb{Z}_p^n$ arbitrarily. 
We need to count the $U_0$ and find $v$ so that $M \in \overline{\mathrm{O}}(S,\mathbb{Z}_{p^2})$. 
Note that $v = py$ where $M_0^ty \equiv (d/p)xw\ \mathrm{mod}\ p$, and so $v$ is determined from $x, w$ and $N_0$. 
Thus it suffices to count $U_0$ such that 
$$(N_0+pU_0)^t(N_0+p^bU_0)-dww^t\ \equiv\ I \ \mathrm{mod} \ p^2.$$
The above congruence is equivalent to 
$$U_0^tN_0 + N_0^tU_0\ \equiv\ \frac{I-N_0^tN_0}{p} + (d/p) ww^t \ \mathrm{mod}\ p.$$
Let $V_0 = N_0^tU_0$.  Then $U_0 = (N_0^t)^{-1}V_0$. 
The congruence
$$V_0^t + V_0\ \equiv\ \frac{I-N_0^tN_0}{p} + (d/p) ww^t \ \mathrm{mod}\ p$$
determines the diagonal of $V_0$ and below diagonal entries from the above diagonal entries. 
Hence there are $p^{n(n-1)/2}$ choices for $V_0$ and so for $U_0$ given $w$ and $x$. 
Thus there are $2p^{n(n-1)/2 + n}$ possibilities for $M \in \overline{\mathrm{O}}(S,\mathbb{Z}_{p^2})$ 
for each $N_0 \in \mathrm{O}(n, \mathbb{Z}_p)$. 
\end{proof}

\begin{theorem}  
The values of $E_{p^2}(S)$ and $E_{p^3}(S)$ are given by
\begin{eqnarray*}
E_{p^2}(S) & =  & 2p^{\frac{(n+1)(n+2)}{2}}|\mathrm{O}(n, \mathbb{Z}_p)|, \\
E_{p^3}(S) & = & p^{\frac{n(n+1)}{2}}E_{p^2}(S).
\end{eqnarray*}
\end{theorem}
\begin{proof}
By Lemmas 4 and 6, we have that 
\begin{eqnarray*}
E_{p^2}(S) & =  & p^{n+1} |\overline{\mathrm{O}}(S,\mathbb{Z}_{p^2})| \\
& = & p^{n+1} 2p^{\frac{n(n+1)}{2}}|\mathrm{O}(n, \mathbb{Z}_p)| \ \ = \ \ 2p^{\frac{(n+1)(n+2)}{2}}|\mathrm{O}(n, \mathbb{Z}_p)|.
\end{eqnarray*}
By Lemmas 4 and 5, we have that 
\begin{eqnarray*}
E_{p^3}(S) & =  & p^{n+1} |\overline{\mathrm{O}}(S,\mathbb{Z}_{p^3})| \\
& = & p^{n+1} p^{\frac{n(n+1)}{2}}|\mathrm{O}(S, \mathbb{Z}_{p^2})| \ \ = \ \ p^{\frac{n(n+1)}{2}}E_{p^2}(S). 
\end{eqnarray*}

\vspace{-.2in}
\end{proof}

\begin{corollary} 
We have that 
$$\frac{2p^{\frac{3n(n+1)}{2}}}{E_{p^3}(S)} = \frac{2p^{\frac{2n(n+1)}{2}}}{E_{p^2}(S)} = 
\frac{p^{\frac{(n+1)(n-2)}{2}}}{|\mathrm{O}(n, \mathbb{Z}_p)|}.$$
\end{corollary}

We need the following classical theorem,  see Theorem 172 of Dickson \cite{D}. 

\begin{theorem} 
If $p$ be an odd prime number, then the order of $\mathrm{O}(n, \mathbb{Z}_p)$ is given by 
$$|\mathrm{O}(n, \mathbb{Z}_p)| = \left\{\begin{array}{ll} 2{\displaystyle\prod_{k=1}^{n-1}}p^k-\epsilon_2(k) & n\  \mathrm{odd,} \\
2p^{\frac{n}{2}-1}\left(p^{\frac{n}{2}} - \left(\frac{-1}{p}\right)^{\frac{n}{2}}\right){\displaystyle\prod_{k=1}^{n-2}}
p^k-\epsilon_2(k) & n\ \mathrm{even,}\end{array}\right.$$
where $(-1/p)$ is a Legendre symbol. 
\end{theorem}
\begin{corollary} 
We have that
$$\frac{2p^{\frac{3n(n+1)}{2}}}{E_{p^3}(S)} = {\displaystyle \frac{p^{\frac{n}{2}}+\epsilon_2(n)\left(\frac{-1}{p}\right)^{\frac{n}{2}}}{\displaystyle 2p^{\frac{n+2}{2}}\prod_{k=1}^{[\frac{n}{2}]}(1-p^{-2k})}}.$$
\end{corollary}
\begin{proof}
First assume that $n$ is odd.  By Theorem 3 and Corollary 2,  we have that 
\begin{eqnarray*}
\frac{2 p^{\frac{3n(n+1)}{2}}}{E_{p^3}(S)} 
& = &\frac{p^{\frac{(n+1)(n-2)}{2}}}
{2{\displaystyle\prod_{k=1}^{n-1}p^k-\epsilon_2(k)}} \\
& = & \frac{p^{\frac{(n+1)(n-2)}{2}}}{
{\displaystyle\, 2\prod_{k=1}^{\frac{n-1}{2}}p^{2k-1}(p^{2k}-1)}} \\
& = & \frac{p^{\frac{(n+1)(n-2)}{2}}}{
2{\displaystyle \prod_{k=1}^{\frac{n-1}{2}}p^{4k-1}(1-p^{-2k})}} \\
& = & \frac{p^{\frac{(n+1)(n-2)}{2}}}{
2p^{\frac{(n-1)n}{2}}{\displaystyle \prod_{k=1}^{\frac{n-1}{2}}(1-p^{-2k})}} \ \
 = \ \ \frac{1}{2p{\displaystyle \prod_{k=1}^{\frac{n-1}{2}}(1-p^{-2k})}}.
\end{eqnarray*}
Now assume that $n$ is even.  By Theorem 3 and Corollary 2,  we have that 
\begin{eqnarray*}
\frac{2 p^{\frac{3n(n+1)}{2}}}{E_{p^3}(S)} 
& = &\frac{p^{\frac{(n+1)(n-2)}{2}}}
{2p^{\frac{n}{2}-1}\left(p^{\frac{n}{2}} - \left(\frac{-1}{p}\right)^{\frac{n}{2}}\right){\displaystyle\prod_{k=1}^{n-2}}p^k-\epsilon_2(k)} \\
& = & \frac{p^{\frac{(n+1)(n-2)}{2}}\left(p^{\frac{n}{2}}+\left(\frac{-1}{p}\right)^{\frac{n}{2}}\right)}
{2p^{\frac{n}{2}-1}(p^n - 1){\displaystyle\prod_{k=1}^{n-2}}p^k-\epsilon_2(k)} \\
& = & \frac{p^{\frac{(n+1)(n-2)}{2}}\left(p^{\frac{n}{2}}+\left(\frac{-1}{p}\right)^{\frac{n}{2}}\right)}
{2p^{-\frac{n}{2}}{\displaystyle\prod_{k=1}^{n}}p^k-\epsilon_2(k)} \\
& = & \frac{p^{\frac{(n+1)(n-2)}{2}}\left(p^{\frac{n}{2}}+\left(\frac{-1}{p}\right)^{\frac{n}{2}}\right)}{
{\displaystyle 2p^{-\frac{n}{2}}p^{\frac{n(n+1)}{2}}\prod_{k=1}^{\frac{n}{2}}(1-p^{-2k})}} \ \ 
= \ \ \frac{p^{\frac{n}{2}}+\left(\frac{-1}{p}\right)^{\frac{n}{2}}}{
{\displaystyle 2p^{\frac{n+2}{2}}\prod_{k=1}^{\frac{n}{2}}(1-p^{-2k})}}.
\end{eqnarray*}

\vspace{-.3in}
\end{proof}

\section{The covolume of the group of units of $f^n_d$} 

Let $\Gamma_d^n$ be the discrete group of isometries of hyperbolic $n$-space $H^n$ 
that corresponds to the group of positive units of the quadratic form $f_d^n$.  
In this section, we give an explicit formula for $\mathrm{vol}(H^n/\Gamma_d^n)$. 

From Formula (6) of \cite{RT97}, we have 
\begin{equation}
\prod^n_{k=1}\pi^{-\frac{k}{2}}\Gamma({\textstyle{\frac{k}{2}}}) = 
\left\{\begin{array}{ll} \displaystyle
\prod^{\frac{n-1}{2}}_{k=1}\frac{2(2k-1)!}{(2\pi)^{2k}}  & n \  \mbox{odd}, \\ \displaystyle
\prod^{\frac{n}{2}}_{k=1}\frac{2(2k-1)!}{(2\pi)^{2k}}\cdot \frac{(2\pi)^{\frac{n}{2}}}{(n-1)!!}  &  n \  \mbox{even}.
\end{array}\right.
\end{equation}
By Theorems 12.17  and 12.18 of \cite{A}, we have that 
\begin{equation}
\zeta(2k) = \frac{(2\pi)^{2k}|B_{2k}|}{2(2k)!}
\end{equation}
for every positive integer $k$ where $B_{2k}$ is the $(2k)$th Bernoulli number. 

Define a function $B$ of $n$ by the formula
\begin{equation}
B = \prod^{[\frac{n}{2}]}_{k=1}\frac{|B_{2k}|}{2k}.
\end{equation}
Then we have that 
\begin{equation}
\prod^n_{k=1}\pi^{-\frac{k}{2}}\Gamma({\textstyle{\frac{k}{2}}}) \prod^{[\frac{n}{2}]}_{k=1}\zeta(2k) =
\left\{\begin{array}{ll} 
B & n\  \mbox{odd}, \\ \displaystyle
B\cdot \frac{(2\pi)^{\frac{n}{2}}}{(n-1)!!}  &  n \  \mbox{even}. 
\end{array}\right.
\end{equation}

Define a function $C$ of $n$ and $d$ by the formula
\begin{equation}
C = \cos((n+(-1)^{\frac{d+1}{2}})\pi/4).
\end{equation}
\begin{theorem} 
Let $d$ be an odd,  square-free, positive integer, and let $\Gamma_d^n$ be the discrete group 
of isometries of hyperbolic $n$-space $H^n$ corresponding to the group of positive units of 
the quadratic form $f_d^n$.  The volume of $H^n/\Gamma_d^n$ is given by 
$$\mathrm{vol}(H^n/\Gamma_d^n) = 
\left\{\begin{array}{ll} 
{\displaystyle\frac{d^{\frac{n-1}{2}}B}{2^{n+\omega(d)}}}\big(2^{\frac{n-1}{2}}+C\big)\big(2^{\frac{n+1}{2}}-\left(\frac{D}{2}\right)\big)
\sqrt{d}\cdot L\big(\textstyle\frac{n+1}{2},D\big) & n \ \mbox{odd}, \vspace{.05in} \\
{\displaystyle\frac{B}{2^{\frac{n}{2}+\omega(d)}}}\big(2^{\frac{n}{2}}+2^{\frac{1}{2}}C\big)
{\displaystyle\prod_{p\,|\,d}}\Big(p^{\frac{n}{2}}+\left(\frac{-1}{p}\right)^{\frac{n}{2}}\Big)
\cdot {\displaystyle\frac{(2\pi)^{\frac{n}{2}}}{(n-1)!!}} & n \ \mbox{even}.
\end{array}\right.$$
\end{theorem}
\begin{proof}
First assume that $n$ is odd. From Formulas (1), (4), (14), (24) and Corollaries 1 and 3, we have that 
\begin{eqnarray*}
\mathrm{vol}(H^n/\Gamma_d^n) \hspace{-.08in} 
&\! \! \! = \! \! \! & 
4d^{\frac{n+2}{2}}
\prod^n_{k=1}\pi^{-\frac{k}{2}}\Gamma({\textstyle{\frac{k}{2}}})\,\cdot\,
\lim_{q\to\infty}{2^{\omega(q)}\frac{q^{\frac{n(n+1)}{2}}}{E_q(S)}} \\
&\! \! \! =\! \! \! & 
4d^{\frac{n+2}{2}}
\prod^n_{k=1}\pi^{-\frac{k}{2}}\Gamma({\textstyle{\frac{k}{2}}})\,\cdot\,
\prod_{p\, |\, 2d}\frac{2p^{\frac{3n(n+1)}{2}}}{E_{p^3}(S)}\,\cdot\,
\prod_{p\, \nmid\, 2d} \frac{2p^{\frac{n(n+1)}{2}}}{E_p(S)} \\
&\! \! \! = \! \! \! & 
4d^{\frac{n+2}{2}}B \!
\prod_{p | 2d}\frac{2p^{\frac{3n(n+1)}{2}}}{E_{p^3}(S)}
{\textstyle\big(1-\left(\frac{D}{2}\right)}2^{-(\frac{n+1}{2})}\big)L({\textstyle\frac{n+1}{2}},D) \! \!
\prod_{k=1}^{\frac{n-1}{2}}\prod_{p|2d}(1-p^{-2k}) \\
&\! \! \! = \! \! \! & 
4d^{\frac{n+2}{2}}B
\frac{\big(2^{\frac{n-1}{2}}+C\big)}{2^{\frac{n+3}{2}}2^{\omega(d)}d}
{\textstyle\big(1-\left(\frac{D}{2}\right)}2^{-(\frac{n+1}{2})}\big)L({\textstyle\frac{n+1}{2}},D) \\
&\! \! \! = \! \! \! & 
{\displaystyle\frac{d^{\frac{n-1}{2}}B}{2^{n+\omega(d)}}}\big(2^{\frac{n-1}{2}}+C\big)
{\textstyle\big(2^{\frac{n+1}{2}}-\left(\frac{D}{2}\right)\big)}
\sqrt{d}\cdot L\big(\textstyle\frac{n+1}{2},D\big).
\end{eqnarray*}
Now assume that $n$ is even.  From Formulas (1), (4), (7), (24) and Corollaries 1 and 3, we have that
\begin{eqnarray*}
\mathrm{vol}(H^n/\Gamma_d^n) \hspace{-.08in} 
&\! \! \! =\! \! \! & 
4d^{\frac{n+2}{2}}
\prod^n_{k=1}\pi^{-\frac{k}{2}}\Gamma({\textstyle{\frac{k}{2}}})\,\cdot\,
\prod_{p\, |\, 2d}\frac{2p^{\frac{3n(n+1)}{2}}}{E_{p^3}(S)}\,\cdot\,
\prod_{p\, \nmid\, 2d} \frac{2p^{\frac{n(n+1)}{2}}}{E_p(S)} \\
&\! \! \! = \! \! \! & 
4d^{\frac{n+2}{2}}B \cdot \frac{(2\pi)^{\frac{n}{2}}}{(n-1)!!} 
\prod_{p\, |\, 2d}\frac{2p^{\frac{3n(n+1)}{2}}}{E_{p^3}(S)} \,\cdot\,
\prod_{k=1}^{\frac{n}{2}}\prod_{p|2d}(1-p^{-2k}) \\
&\! \! \! = \! \! \! & 
4d^{\frac{n+2}{2}}B \cdot \frac{(2\pi)^{\frac{n}{2}}}{(n-1)!!} 
\frac{\big(2^{\frac{n-1}{2}}+C\big)}{2^{\frac{n+3}{2}}}
\prod_{p\, |\, d}\frac{\Big(p^{\frac{n}{2}}+\left(\frac{-1}{p}\right)^{\frac{n}{2}}\Big)}{2p^{\frac{n+2}{2}}} \\
&\! \! \! = \! \! \! & 
{\displaystyle\frac{B}{2^{\frac{n}{2}+\omega(d)}}}\big(2^{\frac{n}{2}}+2^{\frac{1}{2}}C\big)
{\displaystyle\prod_{p\,|\,d}}\Big(p^{\frac{n}{2}}+{\textstyle\left(\frac{-1}{p}\right)^{\frac{n}{2}}}\Big)
\cdot {\displaystyle\frac{(2\pi)^{\frac{n}{2}}}{(n-1)!!}}.
\end{eqnarray*}

\vspace{-.2in}
\end{proof}

\begin{corollary} 
If $n$ is odd, then the Euler characteristic of $\Gamma_d^n$ is zero. 
If $n$ is even, then  the Euler characteristic of $\Gamma_d^n$ is given by
$$\chi(\Gamma_d^n) = {\displaystyle\frac{(-1)^{\frac{n}{2}}B}{2^{\frac{n}{2}+\omega(d)}}}\big(2^{\frac{n}{2}}+2^{\frac{1}{2}}C\big)
{\displaystyle\prod_{p\,|\,d}}\Big(p^{\frac{n}{2}}+{\textstyle\left(\frac{-1}{p}\right)^{\frac{n}{2}}}\Big).$$
\end{corollary}
\begin{proof}
The group $\Gamma_d^n$ is finitely generated, and so $\Gamma_d^n$ has a torsion-free subgroup of finite index 
by Selberg's Lemma. 
If $n$ is odd, then $\chi(\Gamma_d^n) = 0$, since odd dimensional hyperbolic manifolds 
of finite volume have zero Euler characteristic. 
If $n$ is even, then the Gauss-Bonnet Theorem implies that 
\begin{equation}
\mathrm{vol}(H^n/\Gamma_d^n) = (-1)^{\frac{n}{2}}\chi(\Gamma_d^n)\cdot {\displaystyle\frac{(2\pi)^{\frac{n}{2}}}{(n-1)!!}}.
\end{equation}
The result now follows immediately from Theorem 4.
\end{proof}

Using the formulas in Theorem 4,  we can easily compute 
the volume of $H^n/\Gamma_d^n$.  
For example, we have 
\begin{eqnarray*}
\mathrm{vol}(H^3/\Gamma^3_3) & = & \frac{5\sqrt{3}}{64}L(2,-3), \\
\mathrm{vol}(H^3/\Gamma^3_7) & = & \frac{7\sqrt{7}}{64}L(2,-7), \\
\mathrm{vol}(H^5/\Gamma^5_5) & = & \frac{15\sqrt{5}}{2\,048}L(3,5), \\
\mathrm{vol}(H^7/\Gamma^7_7) & = & \frac{49\sqrt{7}}{98\, 304}L(4,-7).  
\end{eqnarray*}
These four volumes were reported on pages 344 and 345 of \cite{J-T}, with some sign differences 
due to different definitions of $\Gamma_d^n$ and $L(s,D)$.  
In \cite{J-T}, the group $\Gamma_d^n$ was denoted by $\Gamma_{-d}^n$, 
and $L$-functions were defined via Legendre symbols rather than Kronecker symbols.

\section{The volumes of Mcleod's polytopes} 

J. Mcleod \cite{M} has determined the reflection subgroup of $\Gamma_3^n$ 
and has shown that it has finite index if and only if $n \leq 13$. 
Let $P^n$ be the hyperbolic  Coxeter fundamental domain of the reflection subgroup of $\Gamma_3^n$ 
determined by Mcleod in \cite{M}.  
Mcleod showed that the reflection subgroup of $\Gamma_3^n$ 
has index equal to the order of the symmetry group $\mathrm{Sym}(P^n)$ of $P^n$; 
moreover, Mcleod determined $\mathrm{Sym}(P)$ for $n \leq 13$. 
Hence, for $n \leq 13$, we have that 
\begin{equation}
\mathrm{vol}(P^n) = |\mathrm{Sym}(P^n)| \mathrm{vol}(H^n/\Gamma^n_3).
\end{equation}
Table 1 lists the volumes of $H^n/\Gamma^n_3$ and $P^n$ for $n \leq 13$ computed using our formulas.

\medskip
\begin{table}  
$\begin{array}{rlcll}
n  & \mathrm{vol}(\Gamma^n_3) & |\mathrm{Sym}(P^n)| & \mathrm{vol}(P^n) & |\mathrm{vol}(P^n)| \vspace{.02in}  \\ 
\hline
\vspace{-.1in} \\
2 & \frac{\pi}{12} 		       		 & 1            & \frac{\pi}{12}  &  2.617993878*10^{-1}  \vspace{.05in} \\ 
3 & \frac{5\sqrt{3}}{64}L(2,-3)  		 & 1            & \frac{5\sqrt{3}}{64}L(2,-3) & 1.057230840*10^{-1} \vspace{.05in} \\
4 &  \frac{\pi^2}{288} 	       		 & 1            & \frac{\pi^2}{288}  & 3.426945973*10^{-2} \vspace{.05in} \\
5 & \frac{\sqrt{3}}{320}L(3,12) 		 & 1 & \frac{\sqrt{3}}{320}L(3,12) & 5.358748797*10^{-3} \vspace{.05in} \\
6 &  \frac{13\pi^3}{604\, 800}     		 & 1            & \frac{13\pi^3}{604\, 800} & 6.664708943*10^{-4} \vspace{.05in} \\
7 & \frac{51\sqrt{3}}{1\, 146\, 880}L(4,-3)  & 1 & \frac{51\sqrt{3}}{1\, 146\, 880}L(4,-3) & 7.240232999*10^{-5} \vspace{.05in} \\
8 &  \frac{697\pi^4}{9\, 144\, 576\, 000}       & 1            & \frac{697\pi^4}{9\, 144\, 576\, 000} & 7.424525364*10^{-6}  \vspace{.05in} \\
9 & \frac{L(5,12)}{716\, 800\sqrt{3}}       & 1            & \frac{L(5,12)}{716\, 800\sqrt{3}} & 8.051559421*10^{-7} \vspace{.05in} \\
10 & \frac{341\pi^5}{987\, 614\, 208\, 000}  & 2    & \frac{341\pi^4}{493\, 807\, 104\, 000} & 2.113228256*10^{-7} \vspace{.05in} \\
11 & \frac{403 L(6,-3)}{12\, 918\, 456\, 320\sqrt{3}}  & 2 & \frac{403 L(6,-3)}{6\, 459\, 228\, 160\sqrt{3}}  & 3.546550442*10^{-8} \vspace{.05in} \\
12 & \frac{50\, 443\pi^6}{12\, 428\, 137\, 193\, 472\, 000}  & 2  & \frac{50\, 443\pi^6}{6\, 214\, 068\, 596\, 736\, 000} & 7.804122909*10^{-9} \vspace{.05in} \\
13 & \frac{691 L(7,12)}{344\, 408\, 064\, 000\sqrt{3}}   & 4   & \frac{691 L(7,12)}{86\, 102\, 016\, 000\sqrt{3}} & 4.633381297*10^{-9} \vspace{.05in} 
\end{array}$

\medskip
\caption{The volumes of Mcleod's hyperbolic Coxeter polytopes}
\end{table}

The polygon $P^2$ is $30^\circ - 45^\circ$ right triangle, and so we can compute the area of $P^2$ 
by the classical angle defect formula 
$$\mathrm{vol}(P^2) = \pi - \left(\frac{\pi}{2} - \frac{\pi}{4} -\frac{ \pi}{6}\right) = \frac{\pi}{12}.$$

The polyhedron $P^3$ is an orthotetrahedron with angles $\frac{\pi}{6}, \frac{\pi}{3},\frac{\pi}{4}$. 
By Theorem 10.4.5 and the duplication formulas 10.4.9 and 10.4.10 in \cite{R}, we have that 
$$\mathrm{vol}(P^3) = \frac{5}{16}\russianL\Big(\frac{\pi}{3}\Big)$$
where $\russianL(\theta)$  is the Lobachevsky function.  
By Formula 2 in Milnor \cite{Milnor}, we have that
\begin{equation}
\russianL(\theta) = \frac{1}{2}\sum_{k=1}^\infty\sin(2k\theta)k^{-2}.
\end{equation}
Observe that 
\begin{equation}
\sin(2k\pi/3) = \frac{\sqrt{3}}{2}\left(\frac{-3}{k}\right).
\end{equation}
Hence we have that 
$$\mathrm{vol}(P^3) = \frac{5\sqrt{3}}{64} L(2,-3).$$

For even $n$, we can compute the volume of the polytope $P^n$ using the Gauss-Bonnet Theorem 
which implies that 
\begin{equation}
\mathrm{vol}(P^n) = (-1)^{\frac{n}{2}}\chi(\Rho^n)\cdot{\displaystyle\frac{(2\pi)^{\frac{n}{2}}}{(n-1)!!}}
\end{equation}
where $\chi(\Rho^n)$ is the Euler characteristic of the Coxeter group $\Rho^n$ defined by $P^n$.  
By Proposition 3 of Chiswell \cite{Chiswell}, we have that 
\begin{equation}
\chi(\Rho^n) = \sum_\Delta \frac{(-1)^{|\Delta|}}{|C_\Delta|}
\end{equation}
where the sum is over all subgraphs $\Delta$ of the Coxeter graph of $\Rho^n$ 
such that the Coxeter group $C_\Delta$ defined by $\Delta$ is finite. 
Here $|\Delta|$ is the number of vertices of $\Delta$. 
From the Coxeter graph of $\Rho^n$ given in Mcleod \cite{M}, 
we computed the volumes of $P^n$ for all even $n \leq 12$ 
using Chiswell's Euler characteristic formula and the Gauss-Bonnet Theorem. 
The volumes agree with the volumes given in Table 1. 

For odd $n$, we computed the volume of $P^n$ by numerical integration.  
We found very close agreement with the numerical values in Table 1 in dimensions 3, 5, 7, 9. 
We found close agreement in dimensions 11 and 13, but the accuracy deteriorated 
to a 1\% discrepancy in dimension 11 and a 3\% discrepancy in dimension 13, 
which is not bad considering we numerically integrated an $n$-fold integral for $n = 11, 13$. 
After the agreement of all these alternate calculations of the volume of $P^n$ for all $n \leq 13$, 
we are confident with the correctness of our 
volume formulas derived using Siegel's analytic theory of quadratic forms.

\section{A commensurability ratio} 

M. Belolipetsky and V. Emery \cite{B-E} proved that for each odd dimension $n \geq 5$ there is a unique 
orientable, noncompact, arithmetic, hyperbolic $n$-orbifold $H^n/\Delta_n$ of smallest volume.  
If $n \equiv 3\ \mathrm{mod}\ 4$, they proved that the volume of $\Delta_n$ is given by
\begin{equation}
\mathrm{vol}(H^n/\Delta_n) = \frac{3^{\frac{n}{2}}}{2^{\frac{n-1}{2}}} L({\textstyle\frac{n+1}{2}},-3) \prod_{k=1}^{\frac{n-1}{2}}
\frac{(2k-1)!}{(2\pi)^{2k}}\zeta(2k).
\end{equation}
In terms of Bernoulli numbers, the above formula can be rewritten as 
\begin{equation}
\mathrm{vol}(H^n/\Delta_n) = \frac{3^{\frac{n}{2}}B}{2^{n-1}} L({\textstyle\frac{n+1}{2}},-3).
\end{equation}

In  \cite{E}, V. Emery mentioned that $\Delta_n$ is commensurable to $\Gamma_3^n$ when $n \equiv 3\ \mathrm{mod}\ 4$. 
We now compute the ratio $\mathrm{vol}(H^n/\Gamma^n_3)/\mathrm{vol}(H^n/\Delta_n)$ for $n \equiv 3\ \mathrm{mod}\ 4$. 
By Theorem 4, we have that 
\begin{equation}
\mathrm{vol}(H^n/\Gamma^n_3) = 
\frac{3^{\frac{n}{2}}B}{2^{n+1}}\big(2^{\frac{n-1}{2}}+(-1)^{\frac{n+1}{4}}\big)
\big(2^{\frac{n+1}{2}}+1\big)L({\textstyle\frac{n+1}{2}},-3).
\end{equation}
Therefore, we have
\begin{equation}
\frac{\mathrm{vol}(H^n/\Gamma^n_3)}{\mathrm{vol}(H^n/\Delta_3)} = 
\frac{1}{4}\big(2^{\frac{n-1}{2}}+(-1)^{\frac{n+1}{4}}\big)\big(2^{\frac{n+1}{2}}+1\big).
\end{equation}

Although $\Delta_n$ is considered in \cite{B-E} only for $n \geq 5$,  
we can define $\Delta_3$ to be $\mathrm{PGL}(2,\mathcal{O}_3)$ 
where $\mathcal{O}_3$ is the ring of integers of $\mathbb{Q}(\sqrt{-3})$. 
Then $\Delta_3$ is arithmetic 
and Meyerhoff \cite{M85} proved that $H^3/\Delta_3$ has minimum 
volume among all orientable, noncompact, hyperbolic 3-orbifolds.  
The group $\Delta_3$ is the orientation preserving subgroup of a hyperbolic Coxeter group of type $[3,3,6]$. 
The above formula for $\mathrm{Vol}(H^3/\Delta_3)$ gives the correct volume for $H^3/\Delta_3$. 
Hence the above formula for $\mathrm{vol}(H^3/\Gamma^3_3)/\mathrm{vol}(H^3/\Delta_3)$ gives 
the correct value $5/4$. 

The groups $\Gamma^3_3$ and $\Delta_3$ are commensurable. 
The relationship between $\Gamma^3_3$ and $\Delta_3$ is explained by the commensurability 
diagram on page 130 of Johnson et al.\ \cite{J-T2}. 
The group $\Gamma^3_3$ is a hyperbolic Coxeter group of type $[4,3,6]$. 
The Coxeter group $[4,3,6]$ has a Coxeter subgroup of type $[6,3^{1,1}]$ of index 2.
Hence the orientation preserving subgroup of $[6,3^{1,1}]$ has index 4 in $[4,3,6]$. 
Now the Coxeter tetrahedron of type $[6,3^{1,1}]$ can be subdivided into 5 copies of 
the Coxeter tetrahedron of type $[3,3,6]$. See Figure 1. 
Hence $[6,3^{1,1}]$ is conjugate to a subgroup of $[3,3,6]$ of index 5. 
Therefore, the orientation preserving subgroup of $[6,3^{1,1}]$ is conjugate 
to a subgroup of $\Delta_3$ of index 5. 
Thus, the ratio $5/4$ faithfully represents the commensurability relationship between $\Gamma^3_3$ and $\Delta_3$, 
that is, $\Gamma^3_3$ has a subgroup of index 4 that is conjugate to a subgroup of $\Delta_3$ of index 5. 

\begin{figure}
\begin{center}
\includegraphics[viewport=0.75in 0.75in 3.25in 3.25in]{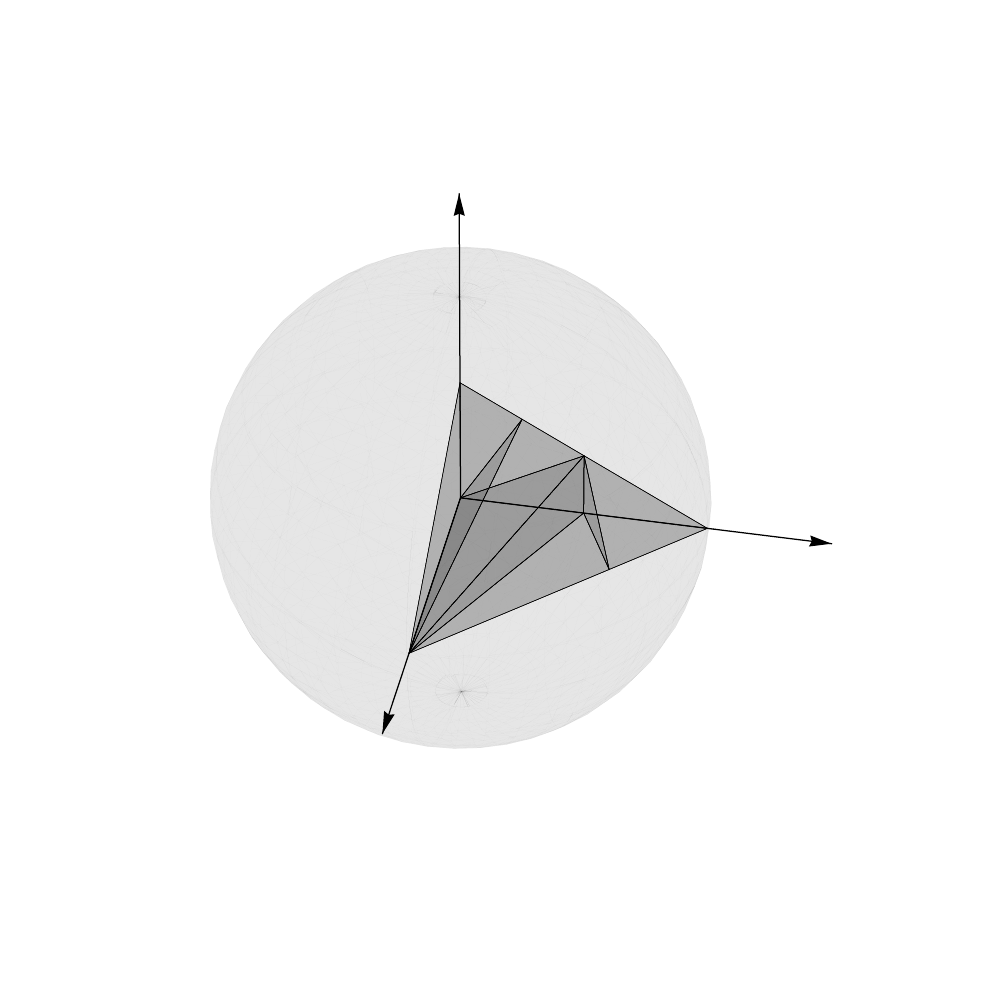}
\end{center}
\caption{The subdivision of the Coxeter tetrahedron of type $[6,3^{1,1}]$ into 5 copies of the Coxeter tetrahedron of type $[3,3,6]$ 
shown in the projective ball model of hyperbolic 3-space with the coordinate axes shown.}  
\end{figure}

We now turn our attention to dimension 7. 
Our volume ratio formula gives that 
$$\mathrm{vol}(H^7/\Gamma^7_3)/\mathrm{vol}(H^7/\Delta_7) = 153/4.$$
By the main theorem of Hild \cite{H}, 
the group $\Delta_7$ is a subgroup of index two of a discrete group $\hat\Delta_7$ of isometries of $H^7$ 
such that $\Delta_7$ is the orientation preserving subgroup of $\hat\Delta_7$. 
The group $\hat\Delta_7$ is generated by a Coxeter subgroup $\bar \Delta_7$ of index 2 of type $[3^{3,2,2}]$ and an involution 
$\sigma$ that acts as the mirror symmetry of the Coxeter system defining $\bar \Delta_7$. 
The involution $\sigma$ is orientation preserving, since $\sigma$ transposes two pairs of Coxeter generators of $\bar\Delta_7$. 
Hence $\Delta_7$ is generated by the orientation preserving subgroup of $\bar\Delta_7$ and $\sigma$. 

We explicitly defined the group $\bar\Delta_7$ to be the group generated by the reflections in the sides 
of the hyperbolic Coxeter 7-simplex $\Delta$ whose sides have Lorentz normal unit vectors listed in the columns 
of the matrix
$$\left(\begin{array}{rrrrrrrr}
0                   &                  0 &  0                 &                  0 &  0                &  -\frac{1}{2} & 0                  & 1 \vspace{.035in}\\
0                   & -\frac{1}{2} & 0                  &                  0 &  0                &   \frac{1}{2}   & 0                  & \frac{1}{2} \vspace{.035in}  \\
0                   &  \frac{1}{2} & 0                  & -\frac{1}{2} &  0                 &  \frac{1}{2}    & 0                 & \frac{1}{2} \vspace{.035in}   \\
0                   &                  0 & 0                  & \frac{1}{2} &  0                 &  0                    & -1                & \frac{1}{2} \vspace{.035in} \\
0                   &                  0 & 0                  &  \frac{1}{2} & -1                 &  \frac{1}{2} &  0                & 0  \vspace{.035in}\\
0                   &  \frac{1}{2} & -1                 &  \frac{1}{2}  & 0                 &  0                 & 0                 &  0 \vspace{.035in}\\
-1                  &  \frac{1}{2} &  0                 &                   0 &                0  &  0                 & 0                &  0  \\
0                   &                  0 &  0                 &                  0  &                0  &  0                 & 0               &  \frac{\sqrt{3}}{2}
\end{array}\right).$$
The involution $\sigma$ acts as a symmetry of $\Delta$. 
The Lorentzian matrix of $\sigma$ is 
$$\left(\begin{array}{rrrrrrrr}
-2                  &                  0 &  0                 &                  0 &  0                &  0                   & 0                  & \sqrt{3}  \\
0                   &                  1 & 0                  &                  0 &  0                &  0                   & 0                  & 0  \\
0                   &                  0 & 1                  &                  0 &  0                 & 0                   & 0                 & 0   \\
0                   &                  0 & 0                  &                  0 &  1                 &  0                  & 0                & 0  \\
0                   &                  0 & 0                  &                  1 &  0                 &  0                 &  0                & 0  \\
0                   &                  0&  0                  &                  0 & 0                 &  1                  & 0                 &  0 \\
0                   &                  0 & 0                  &                  0 & 0                 &  0                 & 1                 &  0  \\
-\sqrt{3}       &                  0 & 0                  &                  0 & 0                  &  0                 & 0               &  2
\end{array}\right).$$

In 1996, we discovered that $\Gamma_3^7$ is a Coxeter reflection group defined by the Coxeter graph 
in Figure 1 of \cite{M}, for $n = 7$, using Vinberg's algorithm \cite{V}. 
We explicitly defined the group $\Gamma_3^7$ to be the group generated by the reflections in the sides 
of the hyperbolic Coxeter 7-dimensional polytope $P^7$ whose sides have Lorentz normal unit vectors listed in the columns 
of the matrix
$$\left(\begin{array}{rrrrrrrrr}
-\frac{1}{\sqrt{2}}  &                  0            &                0             &                0             &  0                            & 0                            & 0  & \sqrt{\frac{3}{2}} & \frac{1}{\sqrt{2}} \vspace{.03in} \\
 \frac{1}{\sqrt{2}}  & -\frac{1}{\sqrt{2}}  &                0             &                0             &  0                            & 0                            & 0  & 0                           & \frac{1}{\sqrt{2}} \vspace{.035in} \\
0                             &  \frac{1}{\sqrt{2}}  & -\frac{1}{\sqrt{2}} &                0             &  0                            &  0                           & 0  & 0                           & \frac{1}{\sqrt{2}}\vspace{.035in} \\
0                             &                  0            & \frac{1}{\sqrt{2}}  & -\frac{1}{\sqrt{2}} &  0                            &  0                           & 0 & 0                            & \frac{1}{\sqrt{2}} \vspace{.035in} \\
0                             &                  0            &                0            &  \frac{1}{\sqrt{2}}  & -\frac{1}{\sqrt{2}}  &  0                           & 0  & 0                           & \frac{1}{\sqrt{2}} \vspace{.035in} \\
0                             &                  0            &                0            &                 0             & \frac{1}{\sqrt{2}}   & -\frac{1}{\sqrt{2}} & 0  &  0                          & 0 \vspace{.03in} \\
0                             &                  0            &                0            &                 0             &   0                           & \frac{1}{\sqrt{2}}  & -1 & 0                           & 0 \\
0                             &                  0            &                0            &                 0             &   0                           &  0                           & 0 & \frac{1}{\sqrt{2}}  & \sqrt{\frac{3}{2}} 
\end{array}\right).$$
In 1996, we discovered by a coset enumeration program that $\bar\Delta_7\cap \Gamma_3^7$ has index 2295 in $\bar\Delta_7$ and index 60 in 
$\Gamma_3^7$. The involution $\sigma$ is an element of $\Gamma^7_3$ and so $\Delta_7\cap \Gamma_3^7$ has index 2295 in $\Delta_7$ and 
index 60 in $\Gamma^7_3$. 
The commensurability ratio is $2295/60 = 153/4$. 
We were disappointed that we could not find representations of $\Delta_7$ and $\Gamma_3^7$ 
such that $\Delta_7\cap\Gamma_3^7$ has index 153 in $\Delta_7$ and index 4 in $\Gamma_3^7$, 
but we suspect that the representations that we found give the smallest possible indices. 

In 1996, we found a discrete group $\Gamma$ of isometries of $H^7$ that corresponds to the group of positive units 
of the quadratic form defined by the diagonal matrix $\mathrm{diag}(1,1,1,1,1,1,3,-1)$ 
such that $\bar\Delta_7\cap\Gamma$ has index 119 in $\bar\Delta_7$ and index 4 in $\Gamma$. 
Moreover $\mathrm{vol}(H^7/\Gamma)/\mathrm{vol}(H^7/\bar\Delta_7) = 119/4$,  
and so the ratio 119/4 faithfully represents the commensurability relationship between $\bar\Delta_7$ and $\Gamma$. 
This computation was reported on page 345 of \cite{J-T}.  
The story of the computation of the volume of $H^7/\Gamma$ will have to wait for another day.

\end{document}